\documentclass[12pt]{amsart}
\usepackage{amsmath}
\usepackage{amsthm, color}
\usepackage{amsfonts}
\usepackage{amssymb}
\usepackage{mathrsfs}
\usepackage[margin=1in]{geometry}
\usepackage{setspace}
\usepackage{tikz}
\usetikzlibrary{matrix, arrows}
\usetikzlibrary{decorations.pathmorphing,shapes}
\usepackage{tikz-cd}
\usepackage[inline]{enumitem}
\usepackage[colorlinks, breaklinks]{hyperref}
\hypersetup{linkcolor=black,citecolor=blue,filecolor=black,urlcolor=black} 
\usepackage{pdflscape}
\usepackage{array}
\usepackage{adjustbox}
\usepackage{tabularx}
\usepackage{multirow}
\usepackage{multicol}
\usepackage{arydshln}
\usepackage{todonotes}
\usepackage{ifthen}
\usepackage{ stmaryrd } 
\usepackage{comment} 

\renewcommand{\a}{\mathfrak{a}}

\newcommand{\abs}[1]{\left\vert #1 \right\vert}

\newcommand{\p}{\mathfrak{p}}

\newcommand{\dd}{\partial}

\renewcommand{\emptyset}{\varnothing}
\renewcommand{\epsilon}{\varepsilon}

\newcommand{\Iff}{\Leftrightarrow}

\newcommand{\ImpliedBy}{\Leftarrow}
\newcommand{\Implies}{\Rightarrow}

\DeclareMathOperator{\init}{in}				
\newcommand{\iso}{\cong}
\DeclareMathOperator{\Ker}{Ker}

\newcommand{\kk}{\Bbbk}					        

\renewcommand{\phi}{\varphi}

\DeclareMathOperator{\reg}{reg}			

\renewcommand{\setminus}{\smallsetminus}

\DeclareMathOperator{\Sym}{Sym}
 
\newcommand{\term}[1]{\textsf{\textbf{#1}}}			

\DeclareMathOperator{\Tor}{Tor}

\renewcommand{\vec}[1]{\mathbf{#1}}
\newcommand{\ZZ}{\mathbb{Z}}

\def\l{\label}

\newcommand{\graph}[2]{
	\begin{tikzpicture}          
	\newcommand*\points{#1}     
	\newcommand*\edges{#2}          
	\newcommand*\scale{0.75}          
	\foreach \x/\y/\z/\w in \points {
		\draw[fill = black!50] (\scale*\x,\scale*\y) circle [radius = 0.1] node[label = {[label distance = 0.05 cm]\w: $\z$}] (\z) {}; 
	}
	\foreach \x/\y in \edges { \draw (\x) -- (\y); }      
	\end{tikzpicture}
}

\newcommand{\unlabeledgraph}[2]{
	\begin{tikzpicture}          
	\newcommand*\points{#1}     
	\newcommand*\edges{#2}          
	\newcommand*\scale{0.75}          
	\foreach \x/\y/\z/\w in \points {
		\draw[fill = black!50] (\scale*\x,\scale*\y) circle [radius = 0.1] node[label = {[label distance = 0.05 cm]\w: }] (\z) {}; 
	}
	\foreach \x/\y in \edges { \draw (\x) -- (\y); }      
	\end{tikzpicture}
}

\newcommand{\edgelabeledgraph}[2]{
	\begin{tikzpicture}[scale = 1.2]          
	\newcommand*\points{#1}     
	\newcommand*\edges{#2}          
	\newcommand*\scale{0.75}          
	\foreach \x/\y/\z/\w in \points {
		\draw[fill = black!50] (\scale*\x,\scale*\y) circle [radius = 0.1] node[label = {[label distance = 0.05 cm]\w: $\z$}] (\z) {}; 
	}
	\foreach \x/\y/\p/\a/\l in \edges { \ifthenelse{\x = \y}{\draw (\x) to[out = \a -20, in = \a + 20, looseness = 50]  node [pos = \p, label =\a: \textcolor{blue}{\l}] {} (\y); }{
    \draw (\x) -- (\y) node [pos = \p, label =\a: \textcolor{blue}{\l}] {};}
    }
	\end{tikzpicture}
}

\newtheorem{thm}{Theorem}[section]
\newtheorem{lemma}[thm]{Lemma}
\newtheorem{prop}[thm]{Proposition}
\newtheorem{cor}[thm]{Corollary}

\newtheorem*{main-thm}{Main Theorem}
\newtheorem{theorem}[thm]{Theorem}

\theoremstyle{definition}
\newtheorem{defn}[thm]{Definition}
\newtheorem{example}[thm]{Example}
\newtheorem{remark}[thm]{Remark}

\numberwithin{equation}{section}
\numberwithin{table}{section}
\numberwithin{figure}{section}

\title{Koszul binomial edge ideals}
\date{\today}
\author[A. LaClair]{Adam LaClair}
\address{University of Nebraska-Lincoln, Department of Mathematics, Lincoln, NE, USA}
\email{alaclair2@unl.edu}
\author[M. Mastroeni]{Matthew Mastroeni}
\address{SUNY Polytechnic Institute, Department of Applied Mathematics, Utica, NY, USA}
\email{mastromn@sunypoly.edu}
\author[J. McCullough]{Jason McCullough}
\address{Iowa State University, Department of Mathematics, Ames, IA, USA}
\email{jmccullo@iastate.edu}
\author[I. Peeva]{Irena Peeva}
\address{Cornell University, Department of Mathematics, Ithaca, NY, USA}
\email{ivp1@cornell.edu}

\begin{document}



\begin{abstract} As the binomial edge ideal of a graph is always generated by homogeneous quadratic polynomials corresponding to the edges of the graph, the question of when a binomial edge ideal defines a Koszul algebra has been studied by many authors ever since the class of ideals was first defined.  Several partial results are known, including a characterization of those binomial edge ideals that possess a quadratic Gr\"obner basis.  However, a complete characterization of the graphs determining Koszul binomial edge ideals has remained elusive.  Inspired by our recent work characterizing when the graded M\"obius algebras of graphic matroids are Koszul, we answer the question once and for all by proving that a graph defines a Koszul binomial edge ideal if and only if it is strongly chordal and claw-free.
\end{abstract}

\maketitle

\section{Introduction}

\noindent
Throughout the paper, $\kk$ stands for a field. 
Consider a polynomial ring $S=\kk[x_1,\dots ,x_n]$ standard graded by $\deg(x_i)=1$ for each $i$, and let $I$ be a graded ideal.
The graded $\kk$-algebra $Q=S/I$ is \term{Koszul} if   $\kk$ has a linear  free resolution over $Q$, that is,
the entries in the matrices of the differential in the resolution are linear forms; in this case, we also say that the ideal $I$ is \term{Koszul}.
Koszul algebras were formally defined by Priddy \cite{Pr} (who considered a more general situation which includes the non-commutative case as well).  
They appear in many areas of algebra, geometry and topology. Two of their exceptionally nice properties are that
the minimal free resolution of $\kk$ can be explicitly 
described as the generalized Koszul complex (see \cite[8.12]{MP}),
and there is an elegant formula relating the Hilbert function of $Q$ and the Poincar\'e series of $\kk$ (which is the generating function of the minimal free resolution of $\kk$ over $Q$), see \cite[8.5]{MP}.
Another important property of commutative Koszul algebras is that they are the algebras over which every finitely generated module has finite regularity \cite{AE,AP}.
The papers
 \cite{Co}, \cite{Koszul:algebras:and:regularity}, \cite{Fr}, and \cite{PP} provide excellent overviews of Koszul algebras.
   It is well-known that if $I$ is Koszul then it is generated by quadrics; in this case, we say that $Q$ and $I$ are  \term{quadratic}.

 The Koszul property has been studied for quadratic graded $\kk$-algebras coming from graphs. Let $G$ be a graph on vertex sex $[n] = \{1, 2, \dots, n\}$.  Throughout this paper, all graphs are assumed to be simple, that is, with no loops and no double edges.
 There are various constructions that produce a quadratic graded $\kk$-algebra starting from $G$; we discuss some of them: The monomial edge ideal
 $N_G=\big(x_ix_j\,|\, \{i,j\} \ \hbox{ is an edge in } G\,\big)$ is generated by quadratic monomials, and it is always Koszul.
One may consider the toric edge ring $A_G = \kk[x_ix_j\,|\, \{i,j\} \ \hbox{ is an edge in } G\,]$, for which Hibi and Ohsugi provide a characterization when $A_G$ is quadratic. Furthermore,  if the graph $G$ is bipartite, 
they prove in \cite{OH2} that $A_G$ is Koszul if and only if it is quadratic, if and only if there is a quadratic Gr\"obner basis. 

We focus on the concept of  binomial edge ideal associated with a graph $G$, which generalizes the classical concept of determinantal ideal generated by the $2$-minors of a $(2\times n)$-matrix of indeterminates. 

\begin{defn}\label{maindef}
The \term{binomial edge ideal} $J_G$ is the ideal 
\[
J_G=\Big(\,x_iy_j-y_ix_j\,\Big|\, \,\{i,j\}\ \hbox{ is an edge in } G\,\Big)\,
\]
in the polynomial ring 
\[
S_G =\kk[x_1,\dots ,x_n,y_1,\dots ,y_n]\,.
\]
We set $R_G = S_G/J_G$. 
We will use this notation throughout the paper.
\end{defn}
Binomial edge ideals were first introduced in \cite{HHHKR} and independently in \cite{Oh} as a generalization of the ideals of minors that appear in algebraic statistics;  
see \cite[Section 4]{HHO}.  Binomial edge ideals also generalize ladder determinantal ideals and ideals of adjacent minors studied in \cite{Co2}.
They have been studied extensively, and the book \cite{HHO} provides an overview of their properties.

It turns out that  the Koszul property of $J_G$ is related to three special graphs, defined in the Figure~\ref{3 graphs}:
the tent, the claw, and the net.
\begin{figure}[htb]
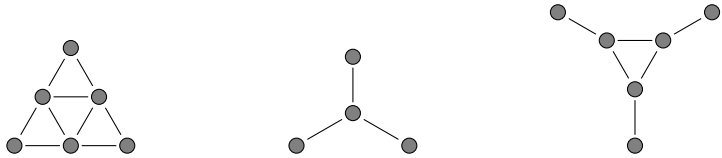

\begin{center}
\unlabeledgraph{
    -1/0/v_3/-90,
    1/0/v_4/-90,
    0/.577/v_1/0,
    0/1.577/v_2/90,
    -5/0/a_2/0,
    -4/0/a_1/0,
    -6/0/a_3/0,
    -4.5/.866/a_4/0,
    -5.5/.866/a_5/0,
    -5/1.732/a_6/0,
      5/1/b_1/0,
    4.5/1.866/b_2/0,
    5.5/1.866/b_3/0,
    5/0/b_4/0,
    3.634/2.366/b_5/0,
    6.366/2.366/b_6/0
    }{
    v_1/v_2,
    v_1/v_3,
    v_1/v_4,
    a_1/a_2,
    a_2/a_3,
    a_1/a_4,
    a_2/a_4,
    a_2/a_5,
    a_3/a_5,
    a_4/a_5,
    a_4/a_6,
    a_5/a_6,
    b_1/b_2,
    b_1/b_3,
    b_2/b_3,
    b_1/b_4,
    b_2/b_5,
    b_3/b_6}
\end{center}
\caption{The tent (left), claw (center), and net (right) graphs} \label{3 graphs}
\label{tent:claw:net}
\end{figure}

\noindent
The study of the Koszul property  of binomial edge ideals was launched by Ene, Herzog, and Hibi in 2014 in \cite{EHH}.  They showed  that if  the binomial edge ideal $J_G$ is Koszul, then the graph $G$ is chordal  \cite{EHH, HHO}, that is, $G$ does not contain an induced cycle of length greater than three. Furthermore, they found that the binomial edge ideal of the claw and the binomial edge ideal of the tent are not Koszul. The ideal $J_G$ is multigraded, and this implies that if $J_G$ is Koszul, then $J_F$ is Koszul for any induced subgraph $F$ of $G$, by \cite[7.46]{HHO}. 
So,  if  $J_G$ is Koszul, then $G$ is chordal, claw-free, and tent-free.
The rather sparse known results on Koszulness of binomial edge ideals
are surveyed in the book \cite{HHO}.

The main result in our paper is that,
 in Section~\ref{Koszul:beis}, we prove the following characterization of Koszulness:

\begin{theorem} \label{koszuledgeintro} \label{mainthm}
The binomial edge ideal $J_G$ is Koszul if and only if the graph $G$ is chordal, claw-free, and tent-free.
\end{theorem}

We also show in Corollary~\ref{strongly:chordal:clawfree} that the above conditions are equivalent to $G$ being strongly chordal and claw-free.

 The usual technique for establishing that an ideal is Koszul is to show that it  
 has a quadratic Gr\"obner basis.  Proving Koszulness in the absence of a  quadratic Gr\"obner basis is very challenging, and there are very few such results known.  For example,  Caviglia and Conca  study the Koszul property of projections of the Veronese cubic surface
  in \cite{CC}, and Hibi and Ohsugi provide such toric rings in \cite{OH1}.
It is known that the binomial edge ideal  $J_G$ has a quadratic Gr\"obner basis if and only if the graph $G$ is closed by \cite[1.1]{HHHKR} and \cite[3.4]{CR}. In this case, $J_G$ has the very rare property that
its graded Betti numbers coincide with those of its lex initial ideal \cite{Pe2}.
   By \cite[7.10]{HHO}, $G$ is closed if and only if it is chordal, claw-free, tent-free, and net-free. Thus, any Koszul binomial edge ideal for which $G$ contains a net as an induced subgraph does not have a quadratic Gr\"obner basis. There are many such examples, starting with the net itself.
 This is in stark contrast with the binomial edge ideals of a pair of graphs, where the two notions are equivalent by \cite{BEI}.

We prove the `if' implication of Theorem~\ref{main:theorem} in Section~\ref{Koszul:beis}. 
The key idea in our proof is that such graphs are strongly chordal and claw-free, which yields orders on both the vertices and edges with special properties. Our proof is an inductive argument based on the properties of this class of graphs.   Although the ``only if" implication was already established by Ene, Herzog, and Hibi  \cite{EHH}, \cite[Section 7.4.1]{HHO}, we give a self-contained, characteristic-free proof that does not rely on computer algebra computations in Section~\ref{obstructions}.  In Section~\ref{thicknet}, we characterize strongly chordal, claw-free graphs with small clique number and show that this characterization does not generalize to graphs with higher clique numbers.

\vglue.8cm
\section{Koszul Binomial Edge Ideals} \label{main:theorem}\label{Koszul:beis}

\noindent The aim of this section is to prove the sufficient condition in our main result Theorem~\ref{koszuledgeintro}:

\begin{theorem} \l{koszuledge}\label{main:thm}
If a simple graph $G$ is chordal, claw-free, and tent-free, then the binomial edge ideal $J_G$ is Koszul.
\end{theorem} 

We first develop the graph machinery necessary. We refer the reader to \cite{We} and \cite{BLS} for more on graph theory.
Throughout, we assume that $G$ is a finite simple graph.
Let $G = (V(G),E(G))$  have a set of vertices $V(G)$, and a set of edges $E(G)$. 
For simplicity, we may write $\{i,j\}\in G$ meaning $\{i,j\}\in E(G)$, or may write $ij$ to denote the edge $\{i,j\}$.
The set of \term{neighbors} of a vertex $i$ in $G$ is the set $N(i)$ of all vertices $j$ such that $\{i, j\}$ is an edge of $G$. The set $N[i] = N(i) \cup \{i\}$ is called the \term{closed neighborhood} of $i$.    A subgraph $H$ of $G$ is \term{induced} if for any $v_i,v_j \in V(H)$, if $v_iv_j \in E(G)$ then $v_iv_j \in E(H)$. The complete graph contains all possible edges between distinct vertices. A \term{cycle} graph $C_n$ of length $n$ has $n$ vertices $v_1,\ldots,v_n$ and edges $v_iv_{i+1}$ for $1 \le i \le n$, where we identify $v_{n+1}$ with $v_1$.

A \term{chordal graph} is one with no induced cycles of length at least four.
There are several characterizations of chordal graphs; we highlight two characterizations involving vertex orderings and special types of vertices.  

A \term{perfect elimination order} for a graph $G$ is an ordering $v_1, \dots, v_n$ of its vertices such that whenever $v_iv_j, v_iv_k \in E(G)$ and $i < j, k$, we have that $v_jv_k \in E(G)$.  We say that $v$ is a \term{simplicial vertex} of $G$ if $N[v]$ is a clique, that is, the induced subgraph of $G$ on the set $N[v]$ is a complete graph.  In that case, $N[v]$ is  a maximal clique of $G$.  If $G$ has a perfect elimination order $v_1, \dots, v_n$, it is easily seen that $v_i$ is a simplicial vertex of the induced subgraph $G_i = G[v_i, \dots, v_n]$ for all $i$.

\begin{theorem}[\cite{Di}] \label{gr}
For a graph $G$, the following are equivalent:
\begin{enumerate}[label = \textnormal{(\alph*)}]
\item $G$ is chordal.
\item $G$ has a perfect elimination order.
\item Every induced subgraph of $G$ has a simplicial vertex.  
\end{enumerate}
\end{theorem}

We will use the stronger concept of strongly chordal graphs.
A \term{strong elimination order} on the set of vertices of $G$ is a perfect elimination order $v_1, \dots, v_n$ such that whenever $v_iv_k, v_kv_j, v_iv_\ell  \in E(G)$ with $i < k < \ell$ and $i < j$, then we also have $v_jv_\ell \in E(G)$. The graph $G$ is called \term{strongly chordal} if it has a strong elimination order.

Similarly to the case of chordal graphs, strongly chordal graphs can also be characterized by special types of vertices.  A vertex $v$ of $G$ is called \term{simple} if the collection of distinct sets in $\{N[w] \mid w \in N[v]\}$ is totally ordered by inclusion. In particular, every simple vertex is also simplicial.  If $v_1, \dots, v_n$ is a strong elimination order for $G$, it is easily seen that the vertex $v_i$ is a simple vertex of the induced subgraph $G_i := G[v_i, \dots, v_n]$ for all $i$. 

Strongly chordal graphs can also be characterized by forbidden induced subgraphs.
A graph $T$ is an \term{$n$-trampoline} if it has vertices $v_1,\dots, v_n, w_1, \dots, w_n$ for some $n \geq 3$ and edges $v_iv_j$ for all $i \neq j$, $w_iv_i$ for all $i$, $w_iv_{i+1}$ for all $i < n$, and $w_nv_1$.  We note that the $n$-trampoline is sometimes referred to as the $n$-sun graph, and the 3-trampoline is also called the \term{tent} in the literature on binomial edge ideals.  The $4$-trampoline is pictured in Figure \ref{4-trampoline} below. 
\begin{figure}[htb]
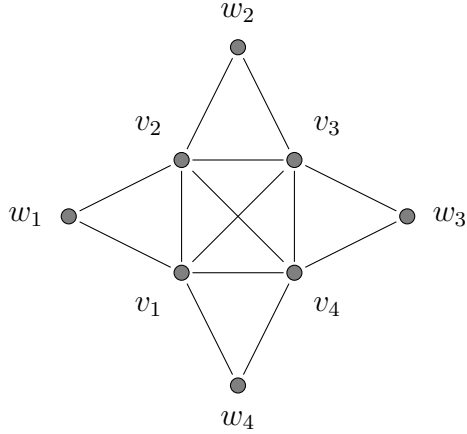

\begin{center}
\graph{
    -1/0/v_1/-120,
    -1/2/v_2/120,
    1/2/v_3/60,
    1/0/v_4/-60,
    -3/1/w_1/180,
    0/4/w_2/90,
    3/1/w_3/0,
    0/-2/w_4/-90
    }{
    v_1/v_2,
    v_2/v_3,
    v_3/v_4,
    v_1/v_4,
    v_2/v_4,
    v_1/v_3,
    w_1/v_1,
    w_1/v_2,
    w_2/v_2,
    w_2/v_3,
    w_3/v_3,
    w_3/v_4,
    v_1/w_4,
    v_4/w_4}
\end{center}
\caption{The 4-trampoline} \label{4-trampoline}
\end{figure}
Setting $w_{i + n} = w_i$, we note that $w_i \in N[v_{i+1}] \setminus N[v_{i+2}]$ and $w_{i+2} \in N[v_{i+2}] \setminus N[v_{i+1}]$ for all $i$ since $n \geq 3$.  Consequently, the sets $N[v_{i+1}]$ and $N[v_{i+2}]$ are incomparable for all $i$. Hence, no vertex of the $n$-trampoline is simple, and so no trampoline graph is strongly chordal.

\begin{theorem}[{\cite[3.3, 4.1]{Fa}}] \label{Farber} \label{strongly:chordal:graphs}
For a graph $G$, the following are equivalent:
\begin{enumerate}[label = \textnormal{(\alph*)}]
    \item $G$ is strongly chordal.
    \item Every induced subgraph of $G$ has a simple vertex.
    \item $G$ is chordal and $n$-trampoline-free for every $n \geq 3$.
\end{enumerate}
\end{theorem}

Applying the above theorem to the graphs that we consider in Theorem~\ref{main:theorem}, we get:

\begin{cor}\label{strongly:chordal:clawfree}
A graph $G$ is chordal, claw-free, and tent-free if and only if it is strongly chordal and claw-free.
\end{cor}

\begin{proof} Note that the 3-trampoline is the {tent}.  Hence, if $G$ is strongly chordal and claw-free, then $G$ is clearly chordal, claw-free, and tent-free.  Conversely, if $G$ is chordal, claw-free, and tent-free, then $G$ is $n$-trampoline-free since every $n$-trampoline with $n > 3$ contains an induced claw, and so, $G$ is strongly chordal. \end{proof}

In the rest of this section, we assume that $G$ is a simple graph with vertex set $[n] = \{1, 2, \dots, n\}$. 

If $i$ is a simplicial vertex of $G$, then $N[i]$ is  the unique maximal clique containing $i$.  This motivates the following definition.

\begin{defn}
Let $G$ be a chordal graph.
\begin{itemize}
    \item A \term{simplicial edge} of a graph $G$ is an edge $e$ that is contained in exactly one maximal clique of $G$ and $G \setminus e$ is chordal.
    \item A simplicial edge is \term{claw-avoiding} if it does not belong to an induced claw of $G$. 
    \item A \term{claw-avoiding perfect edge elimination order} is an ordering $e_1, \dots, e_m$ of the edges of $G$ such that $e_i$ is a claw-avoiding simplicial edge of $G \setminus \{e_1, \dots, e_{i-1}\}$ for every $i$.
\end{itemize}
\end{defn}

\begin{remark}
    We caution the reader that there is at least one other notion of a simplicial edge in the literature.  Dragan \cite{Dr} defines the class of strongly orderable graphs as a common generalization of both strongly chordal and chordal bipartite graphs and characterizes them in terms of edges $e = \{i, j\}$ such that any distinct vertices $k \in N(i)$ and $l \in N(j)$ are adjacent; he calls such an edge a simplicial edge.  For chordal graphs, every simplicial edge in Dragan's sense is a simplicial edge as defined above.  However, our notion of simplicial edge is strictly weaker than Dragan's since the middle edge of a path of length 3 is simplicial in our sense but not Dragan's. 
\end{remark}

\begin{lemma} \label{simplicial:edge:deletion}
Let $G$ be a strongly chordal graph with at least one edge. Then $G$ has a simplicial edge $e$ such that $G \setminus e$ is strongly chordal.  Moreover, one can choose the simplicial edge $e$ so that, if $G$ is claw-free then so is $G \setminus e$.
\end{lemma}

\begin{proof}
We may assume that the vertices of $G$ have been labeled so that $1, 2, \dots, n$ is a strong elimination order.  Then $1$ is a simple vertex of $G$, and we can choose $j$ so that $N[j]$ is maximal in the totally ordered set $\{ N[i] \mid i \in N[1]\}$.  We also note that $N[1]$ is minimal in $\{ N[i] \mid i \in N[1]\}$.  Since $1$ is in particular a simplicial vertex of $G$, it follows that $N[1]$ is the unique maximal clique of $G$ containing $1$.

We claim that the edge $e = \{1, j\}$ is simplicial.  We have already shown that $N[1]$ is the unique maximal clique of $G$ containing e, so it remains to show that $G \setminus e$ is chordal.  If $G \setminus e$ is not chordal, then it contains an induced cycle $C$ of length at least 4.  As $G$ is chordal, $e$ must be the unique chord of $C$, and  so $C$ has length 4.  Hence, $G$ has an induced subgraph of the form 
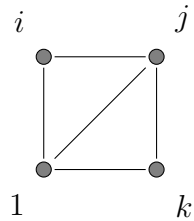
\begin{center}
\graph{
    -1/0/1/-120,
    -1/2/i/120,
    1/2/j/60,
    1/0/k/-60
    }{
    1/i,
    i/j,
    j/k,
    1/k,
    1/j}
\end{center}
which is impossible as $i, k \in N[1]$ and 1 is a simplicial vertex. Hence, $e$ is a simplicial edge of $G$ as claimed.

To see that $G \setminus e$ is still strongly chordal, we note that $N_{G \setminus e}[1] = N_G[1] \setminus \{j\}$ and that $N_{G \setminus e}[i] = N_G[i]$ for all $i \neq 1, j$, so  the set 
\[
\{N_{G \setminus e}[i] \mid i \in N_{G \setminus e}[1]\}
= \{N_G[1] \setminus \{j\}\} \cup \{N_G[i] \mid i \in N_{G}[1] \setminus \{j\} \}
\]
is still totally ordered by inclusion.  Hence, $1$ is still a simple vertex of $G \setminus e$, and the graph $(G \setminus e) \setminus \{1\} = G \setminus \{1\}$ has $2, \dots, n$ as a simple elimination order.  Thus, $1, 2, \dots, n$ is a simple elimination order for $G \setminus e$, so  $G \setminus e$ is strongly chordal by Theorem~\ref{Farber}.

Suppose in addition we know that $G$ is claw-free.  If $G \setminus e$ has an induced claw, then $G$ has an induced subgraph of the form
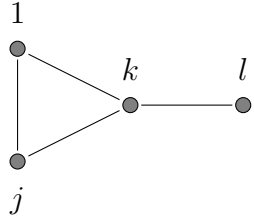
\begin{center}
\graph{-2/1/1/90,0/0/k/90,2/0/l/90,-2/-1/j/-90}{1/k,k/l,j/1,j/k}
\end{center}
But this is impossible since $k \in N[1]$ implies $l \in N[k] \subseteq N[j]$.
\end{proof}

The following is a  refinement of \cite[3.7]{MS14}. 

\begin{prop}
\label{edge:colon}
Let $G$ be a graph  and $e = \{i,j\} \in E(G)$. Denote by $G_{e}$ the graph on $[n]=\{1,\dots ,n\}$ having edge set 
\[
E(G_{e}) = E(G \setminus e) \cup \binom{N_{G \setminus e}(i)}{2} \cup \binom{N_{G \setminus e}(j)}{2}.
\]
For any cycle $C: i = i_0, i_1, \dots, i_s, i_{s+1}= j$ of $G$ containing $e$ and $0 \leq t \leq s$, set  \[ g_{C,t} = y_{i_1}\cdots y_{i_t}x_{i_{t+1}} \cdots x_{i_s}.
\]
Then
\[
J_{G \setminus e} : f_e = J_{G_e} + (g_{C, t} : e \subseteq C \text{ an induced cycle of } G, 0 \leq t \leq s).
\]
\end{prop}

We observe that the graph $G_e$ of the theorem is  the graph obtained from $G \setminus e$ by adding edges to turn the ends of $e$ into simplicial vertices of $G \setminus e$.

\begin{proof}
    The statement of the proposition is precisely \cite[3.7]{MS14} if we omit the condition that the cycles $C$ are induced.  In the statement of that theorem, the monomials are indexed instead by paths from $i$ to $j$ in $G \setminus e$, but such paths correspond bijectively with cycles of $G$ containing $e$.  We will show that each monomial $g_{C, t}$ is divisible by a suitable monomial $g_{C', t'}$ whenever $C$ admits a chord.  
    
    Let $C: i = i_0, i_1, \dots, i_s, i_{s+1} = j$ be any cycle of $G$ containing $e$. If $C$ has a chord $\{i_p, i_q\}$ for some $p < q$, then we have a cycle $C': i = i_0, i_1, \dots, i_p, i_q, i_{q+1}, \dots, i_{s+1} = j$ so that 
\[
g_{C', t} = y_{i_1} \cdots y_{i_t}x_{i_{t+1}} \cdots x_{i_p}x_{i_q} \cdots x_{i_s}
\]
divides $g_{C, t}$ for all $t < p$,
\[
g_{C', p} = y_{i_1} \cdots y_{i_p}x_{i_q} \cdots x_{i_s}
\]
divides $g_{C, t}$ for all $p \leq t < q$, and 
\[
g_{C', p + t - q + 1} = y_{i_1} \cdots y_{i_p}y_{i_q} \cdots y_{i_t}x_{i_{t+1}} \cdots x_{i_s}
\]
divides $g_{C,t}$ for all $t \geq q$.
\end{proof}

\begin{cor}
\label{simplicial:edge:colon}
Let $G$ be a chordal graph, and let $e \in E(G)$ be a claw-avoiding simplicial edge. Let $U$ be the unique maximal clique of $G$ containing $e$. Then 
\begin{align*}
    J_{G\setminus e} : f_e = J_{G \setminus e} + (x_k, y_k \mid k \in U \setminus e).
\end{align*}
\end{cor}

\begin{proof}
By the preceding proposition, we know that $J_{G \setminus e}: f_e = J_{G_e} + L$, where 
\[
L = (g_{C,t} : \text{$C$ is an induced cycle of $G$}, e \subseteq C, \text{ and } 0 \leq t \leq s).\] 
If $e = \{i,j\}$, then for every $k \in U \setminus e$, we have an induced $3$-cycle $i,j,k$ so that 
\[
    (x_k,y_k \mid k \in U\setminus e) \subseteq L.
\]
  On the other hand, every induced cycle of $G$ has length 3 since $G$ is chordal so that $L = (x_k,y_k \mid k \in U\setminus e)$.

Next, we show that $J_{G_e} + L = J_{G \setminus e} + L$.   Since $G \setminus e$ is a subgraph of $G_e$, it is clear that $J_{G \setminus e} + L \subseteq J_{G_e} + L$.  Every additional generator of $J_{G_e}$ but not $J_{G \setminus e}$ is of the form $f_{k, l}$ for some $k,l \in N_{G \setminus e}(i)$ or $k, l \in N_{G \setminus e}(j)$.  If $k, l \in N_{G \setminus e}(i)$, the induced subgraph on $\{i,j,k,l\}$ has as an edge $\{j,k\}$, $\{j,l\}$, or $\{k,l\}$ because $e$ does not belong to an induced claw of $G$.  If $\{k,l\}$ is an edge, then $f_{k, l} \in J_{G \setminus e}$.  On the other hand, if $\{j,l\}$ or $\{j,k\}$ is an edge of $G$, then either $\{i, j, l\}$ or $\{i, j, k\}$ is a clique containing $e$, and so, either $l$ or $k$ belongs to $U$ as $e$ is simplicial, in which case it is clear that $f_{k,l} \in L$.  Either way, we see that $f_{k, l} \in J_{G \setminus e} + L$, and the argument is completely analogous if $k, l \in N_{G \setminus e}(j)$.  Thus, we have $J_{G_e} + L = J_{G \setminus e} + L$ as claimed.
\end{proof}

We are now ready to prove our main theorem. Recall the notation in Definition~\ref{maindef}.

\begin{proof}[Proof of Theorem~\ref{main:thm}]
 We argue by induction on the number of edges of $G$. If $G$ has no edges, then $J_G = 0$ and the statement holds trivially.  So, we may assume that $G$ has at least one edge and that the theorem holds for all claw-free strongly chordal graphs with fewer edges.  In that case, since $G$ is strongly chordal and claw-free, it has a simplicial edge $e$ such that $G \setminus e$ is strongly chordal and claw-free by Lemma~\ref{simplicial:edge:deletion}.    Since $G \setminus e$ is a claw-free strongly chordal graph with fewer edges than $G$, we know that $R_{G \setminus e}$ is Koszul by induction.  We will show that $R_G$ is Koszul as well.

Let $K$ be the unique maximal clique of $G$ containing $e$.  Corollary \ref{simplicial:edge:colon} yields the exact sequence of $R_{G \setminus e}$-modules
\[
0 \to R_H(-2) \stackrel{f_e}{\to} R_{G \setminus e} \to R_G \to 0,
\]
where $H$ is the induced subgraph of $G \setminus e$ obtained by deleting the vertices of $K \setminus e$.  Since $H$ is an induced subgraph of $G \setminus e$, we have that $R_H$ is an algebra retract of $R_{G \setminus e}$;
this means that we have the homogeneous inclusion $i: \ R_H\subset R_{G \setminus e}$ and homogeneous surjection
$\varphi: R_{G \setminus e} \longrightarrow
R_H$ so that $\varphi i=id$ (see \cite{TM} for other applications of algebra retracts in an algebra-graph setting).
Combining this with the assumption that $R_{G \setminus e}$ is Koszul yields that $R_H$ has a linear free resolution as a module over $R_{G \setminus e}$ by \cite[1.4]{combinatorial:pure:subrings}.  Now, combine this with the above short exact sequence, and conclude that $\reg_{R_{G \setminus e}} R_G \leq 1$. It follows that $R_G$ is  Koszul by \cite[Theorem~2]{Koszul:algebras:and:regularity}.
\end{proof}

\begin{cor} \label{strongly:chordal:simplicial:edge:deletion}
    Let $G$ be a chordal graph and $e \in E(G)$ be a claw-avoiding simplicial edge of $G$.  If $G \setminus e$ is strongly chordal and claw-free, then so is $G$.
\end{cor}

\begin{proof}
    By our main theorem~\ref{mainthm}, the ring $R_{G \setminus e}$ is Koszul. Then the proof above implies that $R_G$ is also Koszul.  Applying  our main theorem once more, yields that $G$ is strongly chordal and claw-free.
\end{proof}

\begin{cor}
    A chordal graph is strongly chordal and claw-free if and only if it has a claw-avoiding perfect edge elimination order.
\end{cor}

\begin{proof}
    ($\Implies$): We argue by induction on the number of edges of $G$.  The statement holds trivially if $G$ has no edges, so we may assume that $G$ has at least one edge and that all claw-free strongly chordal graphs with fewer edges have a claw-avoiding perfect edge elimination order.  By Lemma \ref{simplicial:edge:deletion}, there is a simplicial edge $e_1$ of $G$ such that $G \setminus e_1$ is strongly chordal and claw-free.  Note that $e_1$ is  claw-avoiding as $G$ is claw-free.  By induction $G \setminus e_1$ has a claw-avoiding perfect edge elimination order $e_2, \dots, e_m$, and  so $e_1, e_2, \dots, e_m$ is a claw-avoiding perfect edge elimination order for $G$.

    ($\ImpliedBy$): Let $e_1, \dots, e_m$ be a claw-avoiding perfect edge elimination order for $G$. Again, we argue by induction on $m$.  The case $m = 0$ being trivial, we may assume that $m \geq 1$ and that the corollary holds for all chordal graphs having a claw-avoiding perfect edge elimination order with fewer edges.  Since $e_1$ is a claw-avoiding simplicial edge of $G$, we know in particular that $G \setminus e_1$ is chordal, and as $G \setminus e_1$ has claw-avoiding perfect edge elimination order $e_2, \dots, e_m$, it follows by induction that $G \setminus e_1$ is strongly chordal and claw-free.  Thus, Corollary \ref{strongly:chordal:simplicial:edge:deletion} implies that $G$ is also strongly chordal and claw-free.
\end{proof}

\section{Closed graphs and quadratic Gr\"obner bases}\label{closed:graphs}

\noindent
In the next section, we will make use of the concept of closed graphs, whose properties are reviewed here.

A \term{closed order} for a graph $G$ is an ordering $v_1, \dots, v_n$ of its vertices such that whenever $v_iv_j, v_iv_k \in E(G)$ and either $i < j, k$ or $i > j, k$, then $v_jv_k \in E(G)$.  The graph $G$ is a \term{closed graph} if it has a closed order.  It is straightforward to check that every closed order is a strong elimination order.  (There are two cases to consider depending on whether $k < j$ or $j < k$.)  Thus, every closed graph is strongly chordal.

The significance of closed graphs  is that they  characterize when binomial edge ideals have a quadratic Gr\"obner basis.   Hence, the binomial edge ideal of any closed graph is Koszul.  

\begin{theorem}[{\cite[1.1]{HHHKR}\cite[3.4]{CR}}]
For a graph $G$, the following are equivalent:
\begin{enumerate}[label = \textnormal{(\alph*)}]
\item $J_G$ has a quadratic Gr\"obner basis with respect to some monomial order.
\item $J_G$ has a quadratic Gr\"obner basis with respect to some lex order.
\item $G$ is a closed graph.
\end{enumerate}
\end{theorem}

Denoting by
\[ N^{\geq}[v_i] = \{v_j \in N[v_i] \mid j \geq i\} \qquad N^{\leq}[v_i] = \{v_j \in N[v_i] \mid j \leq i\} \]
the sets of neighbors of the vertex $v_i$ that come after or before $i$ respectively in the given order, we can characterize when $G$ has a closed order as follows.

\begin{prop}[{\cite[4.8]{CR}\cite[2.2]{Cox:closed:graphs}}] \label{closed:order}
For a graph $G$, the following are equivalent:
\begin{enumerate}[label = \textnormal{(\alph*)}]
\item $v_1, \dots, v_n$ is closed order for $G$.
\item The sets $N^{\leq}[v_i]$ and $N^{\geq}[v_i]$ are both cliques for all $i$.
\item The set $N^{\geq}[v_i]$ is a clique and $N^{\geq}[v_i] = [v_i, v_{i+r}]$, where $r = \abs{N^{\geq}[v_i]}$, for all $i$.
\end{enumerate}
\end{prop}

A graph $G$ is an \term{intersection graph} if there is a family of sets $\{F_i\}_{i \in V(G)}$ such that $ij \in E(G)$ if and only if $F_i \cap F_j \neq \emptyset$.  A graph $G$ is an \term{interval graph} if it is the intersection graph of a family of open intervals of real numbers.  If the family of intervals can be chosen so that no interval properly contains any other, then $G$ is a \term{proper interval graph}.  A \term{proper interval order} for a graph $G$ is an ordering $v_1, \dots, v_n$ of its vertices such that whenever $v_iv_k \in E(G)$ and $i < j < k$, then $v_iv_j, v_jv_k \in E(G)$.  We note that every proper interval order is a closed order and vice versa by Proposition \ref{closed:order}.

\begin{theorem}[{\cite[2.2]{closed:graphs:are:proper:interval:graphs}\cite[7.10]{HHO}}] \label{theorem:closed:graphs}
For a graph $G$, the following are equivalent:
\begin{enumerate}[label = \textnormal{(\alph*)}]
\item $G$ is a closed graph.
\item $G$ has a proper interval order.
\item $G$ is a proper interval graph.
\item $G$ is chordal, claw-free, net-free, and tent-free.
\end{enumerate}
\end{theorem}

The net (see Figure \ref{tent:claw:net}) is the smallest example of a non-closed graph whose binomial edge ideal is Koszul;  a Koszul filtration argument is given in \cite[7.56]{HHO}.  Our main theorem~\ref{mainthm}  shows that the presence of induced nets is precisely the difference between those binomial edge ideals that are Koszul and those that admit a quadratic Gr\"obner basis.

\section{Obstructions to Koszulness}\label{obstructions}

As discussed in the introduction,  the ``only if" direction of Theorem~\ref{koszuledgeintro} was already shown by Ene, Herzog, and Hibi, namely: 

\begin{theorem}[{\cite{EHH}, \cite[Section 7.4.1]{HHO}}]
If the binomial edge ideal $J_G$ is Koszul, then the graph $G$ is chordal, claw-free, and tent-free.
\end{theorem}

Their argument easily reduces to checking that the binomial edge ideals of $n$-cycles for $n \ge 4$, the claw, and the tent are not Koszul, which they verify by computer.  For completeness, in the three lemmas below, we provide direct proofs that the binomial edge ideals of $n$-cycles for $n \ge 4$, the claw, and the tent are not Koszul in any characteristic. The arguments for the $n$-cycles and the claw are fairly quick, but the reasoning for the tent is longer and more intricate.

Some care about the characteristic is warranted  since D'Al\`{i} and Venturello have shown \cite[5.15]{Koszul:Gorenstein:algebras:from:CM:simplicial:complexes} that for any finite list of primes $P$, it is possible to construct ideals of polynomials with integer coefficients that fail to be Koszul modulo $p$ exactly when $p \in P$.  Moreover, there are examples of graphs $G$ such that the finite resolution of $J_G$ over the ambient polynomial ring can depend on the characteristic of the coefficient field \cite[7.6]{Cohen:Macaulay:Binomial:Edge:Ideals:And:Accessible:Graphs}, \cite[5.9]{Powers:Of:Monomial:Ideals:With:Characteristic:Dependent:Betti:Numbers}.  A priori, the Koszul property of $R_G$ could depend on the characteristic as well, although it follows from Theorem~\ref{main:thm} that this is not the case.

Throughout this section, we set   $f_{i,j} = x_iy_j - x_jy_i$ for any $i$ and $j$ (not necessarily adjacent in $G$); in particular, $J_G$ is generated by the
$f_{i,j}$ corresponding to edges.

We will make frequent use of known restrictions on graded Betti numbers of Koszul algebras. We refer the reader to \cite{Pe} for  background on free resolutions and graded Betti numbers;
we just recall the definition of the graded Betti number $\beta_{i,j}^R(M)$ of a finitely generated graded $R$-module $M$ over a graded $\kk$-algebra $R$:
$$\beta_{i,j}^R(M)=\dim_{\kk}\,\hbox{Tor}^R_i(M,k))_{j}.
$$
The graded Betti numbers may be displayed in the Betti table, which contains $\beta_{i,j}^R(M)$ in the $i,i+j$'th slot.
We will also use the following result of Blum concerning when the symmetric algebra $\Sym_R(M)$  is Koszul.

\begin{prop}[{\cite[3.1]{subalgebras:of:bigraded:Koszul:algebras}}] \label{Koszul:symmetric:algebra}
Let $R$ be a standard graded $\kk$-algebra and $M$ be a finitely generated graded $R$-module such that $\Sym_R(M)$ is Koszul.  Then $R$ is also Koszul, and $\Sym_R^i(M)$ has a linear free resolution over $R$ for all $i$.  In particular, $M$ has a linear free resolution over $R$.
\end{prop}

\begin{lemma} 
Let $C_n$ be the $n$-cycle. Then the ring $R_{C_n}$ is not Koszul.
\end{lemma}

\begin{proof}
 For $n \ge 5$, see the proof in \cite[7.47]{HHO} where it is shown that $J_{C_n}$ has a minimal first syzygy of degree $n$; for $n \ge 5$, if $J_{C_n}$ were Koszul, this contradicts \cite[Lemma 4]{Kem90}.  For $n = 4$, consider the graph $G$ shown below
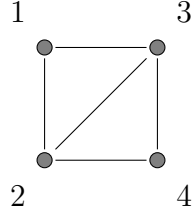
\begin{center}
    \graph{
    -1/0/2/-120,
    -1/2/1/120,
    1/2/3/60,
    1/0/4/-60
    }{
    1/2, 1/3, 2/4, 3/4, 2/3}
\end{center}
and set $e = \{2, 3\}$.  Note that $G \setminus e = C_4$.  We will show that $\beta^S_{2,4}(R_{C_4}) = 9 > \binom{4}{2}$, so  $R_{C_4}$ is not Koszul by \cite[3.4]{Koszul:algebras:defined:by:three:relations}.

Note that $S_G=\kk[x_1, \dots ,x_4, y_1, \dots , y_4].$
First, we observe that that the labeling of $G$ is closed, and so $R_G$ has a quadratic Gr\"obner basis with respect to the lex order with $x_1 > \cdots > x_4 > y_1 > \cdots > y_4$.  In particular, we have that the intial ideal is
\[
\init_>(J_G) = (x_1y_2, x_1y_3, x_2y_3, x_2y_4, x_3y_4),
\]
so it is the edge ideal of a path with 6 vertices.  Using \cite[3.4, 3.5]{edge:ideals:and:DG:algebra:resolutions}, one can  check that the Betti table of $S_G/\init_>(J_G)$ is:
\begin{center}
    \begin{tabular}{c|ccccc}
         & 0 & 1 & 2 & 3 & 4 \\
         \hline
       0 & 1 & -- & -- & -- & -- \\
       1 & -- & 5 & 4 & -- & -- \\
       2 & -- & -- & 3 & 4 & 1
    \end{tabular}
\end{center}
Since $G$ is a closed graph, $\beta_{i,j}^R(R_G) = \beta_{i,j}^{S_G}(S_G/\init_>(J_G))$ for all $i,j$ ( see \cite[1.3]{Pe2}).

To see that $\beta^S_{2,4}(R_{C_4}) = 9$, we note that Theorem \ref{edge:colon} yields the exact sequence
\[
0 \to {S_G}/(x_1, y_1, x_4, y_4)(-2) \xrightarrow{f_{2,3}} R_{C_4} \to R_G \to 0,
\]
from which we obtain the exact sequence
\begin{align*}
0 = \Tor_3^{S_G}(R_G, \kk)_4 &\to
\Tor_2^{S_G}({S_G}/(x_1,y_1, x_4, y_4), \kk)_2 \to \Tor_2^{S_G}(R_{C_4}, \kk)_4 \\
&\to \Tor_2^{S_G}(R_G, \kk)_4 \to \Tor_1^{S_G}({S_G}/(x_1,y_1, x_4, y_4), \kk)_2 = 0.
\end{align*}
Since $x_1,y_1,x_4,y_4$ is a regular sequence on ${S_G}$, we have 
\[
\beta_{2,4}^S(R_{C_4})  = \beta_{2,2}^{S_G}({S_G}/(x_1, y_1, x_4, y_4)) + \beta^{S_G}_{2,4}(R_G) = 6 + 3 = 9
\]
as desired.
\end{proof}

 \begin{lemma} 
Let $G$ be the claw. Then the ring $R_{G}$ is not Koszul.
\end{lemma}

\begin{proof}
 Consider the claw $G$ labeled as shown below 
    \begin{center}
\graph{-2/1/2/90,0/0/1/90,2/0/3/90,-2/-1/4/-90}{1/2,1/3,1/4}
\end{center}
Thus, the binomial edge ideal is $J_G = (f_{1,2}, f_{1,3}, f_{1,4}) \subseteq S_G = \kk[x_1, y_1, \dots, x_4, y_4]$.  If $e = \{1,4\}$, then Theorem \ref{edge:colon} shows that $J_{G \setminus e} : f_e = J_{G_e} = (f_{1,2}, f_{1,3}, f_{2,3})$, since $G \setminus e$ has no induced cycles.  Hence, $J_{G \setminus e} : f_e = I_2(M)$, where
\[
M = \begin{pmatrix}
    x_1 & x_2 & x_3 \\
    y_1 & y_2 & y_3
\end{pmatrix},
\]
which has a Hilbert-Burch resolution.  Applying Theorem \ref{edge:colon} once more to $G' = G \setminus \{4\}$ and the edge $e' = \{1, 3\}$ shows that $(f_{1,2}) : f_{1,3} = J_{G'\setminus e'} : f_{e'} = J_{G'_{e'}} = (f_{1,2})$, so $J_{G \setminus e} = (f_{1,2}, f_{1,3})$ is a complete intersection.  We can construct an explicit chain map lifting the multiplication by $f_{1,4}$ map $S_G/I_2(M)(-2) \to R_{G \setminus e}$ as shown below.
\begin{center}
\begin{tikzcd}[column sep = 13 ex, row sep = 8 ex, ampersand replacement = \&]
    0 \& 
    S_G(-2) \lar \dar[swap]{f_{1,4}}\&
    S_G(-4)^3 \lar[swap]{\begin{pmatrix}
        f_{2,3} & -f_{1,3} & f_{1,2} 
    \end{pmatrix}}
    \dar{\begin{pmatrix}
        f_{2,4} & -f_{1,4} & 0  \\
        -f_{3,4} & 0 & f_{1,4}  
    \end{pmatrix}}
    \&
    S_G(-5)^2 \lar[swap]{\begin{pmatrix}
        x_1 & y_1 \\ x_2 & y_2 \\ x_3 & y_3
    \end{pmatrix}} \dar{\begin{pmatrix}
        -x_4 & -y_4
    \end{pmatrix}}
    \&
    0 \lar
    \\
    0 \& 
    S_G \lar \&
    S_G(-2)^2 \lar{\begin{pmatrix}
        f_{1,3} & f_{1,2} 
    \end{pmatrix}} \&
    S_G(-4) \lar{\begin{pmatrix}
        f_{1,2} \\ -f_{1,3}
    \end{pmatrix}} \&
    0 \lar
\end{tikzcd}
\end{center}
The first column of the middle vertical map comes from the Pl\"ucker relation $f_{1,2}f_{3,4} - f_{1,3}f_{2,4} + f_{2,3}f_{1,4} = 0$ (see \cite[7.2.3]{Bruns:Herzog}).  Hence, the mapping cone of the above chain map yields the minimal free resolution of $R_G$, which has Betti table: \vspace{1 ex} 
\begin{center}
    \begin{tabular}{c|cccc}
         & 0 & 1 & 2 & 3 \\
         \hline
       0 & 1 & -- & -- & -- \\
       1 & -- & 3 & -- & -- \\
       2 & -- & -- & 4 & 2
    \end{tabular}
\end{center}
In particular, because $\beta^S_{2,4}(R_G) = 4 > \binom{3}{2}$, we see that $R_G$ is not Koszul by \cite[3.4]{Koszul:algebras:defined:by:three:relations}.
\end{proof}

\begin{lemma} 
Let  $T$ be the tent. Then the ring $R_T$ is not Koszul.
\end{lemma}

\begin{proof}
Consider the tent $T$ with labeling as shown in Figure~\ref{tent:graph}.
\begin{figure}[htb]
\begin{center}
\graph{
    -1.5/2.598/2/180,
    1.5/2.598/3/0,
    0/0/4/-90,
    -3/0/6/180,
    0/5.196/1/90,
    3/0/5/0
    }{
    2/3,
    3/4,
    2/4,
    1/2,
    1/3,
    5/4,
    3/5,
    2/6,
    4/6}
\end{center}
\caption{}
\label{tent:graph}
\end{figure}

\noindent Setting $G = T \setminus \{6\}$, we note that $R_T \iso \Sym_{R_G}(M)$, where $M$ is the cokernel of the matrix
\[
\dd_1 = \begin{pmatrix}
    -y_2 & -y_4 \\ x_2 & x_4
\end{pmatrix}.
\]
If $R_T$ is Koszul, then $M$ has a linear free resolution over $R_G$ by Proposition \ref{Koszul:symmetric:algebra}.  However, we will show that $M$ has a minimal nonlinear 4-th syzygy, and so the ring $R_T$ is not Koszul.

To see this, we first note that the labeling of $G$ is closed, and  so  $J_G$ has a quadratic Gr\"obner basis with respect to the lex order with $x_1 > \cdots > x_5 > y_1 > \cdots > y_5$.  In particular, we have that the initial ideal is
\[
\init_>(J_G) = (x_1y_2, \, x_1y_3,  \, x_2y_3,  \, x_2y_4,  \, x_3y_4,  \, x_3y_5,  \, x_4y_5).
\]
Thus, the quotient $R_G$ has a monomial basis consisting of all monomials not contained in the above monomial ideal.  We also note that $R_G$ admits a $\ZZ^5$-grading via $\deg x_i = \deg y_i = \vec{e}_i$, where $\vec{e}_i$ is the $i$-th standard basis vector of $\ZZ^5$, and $\Ker \dd_1$ is homogeneous with respect to this grading. We exploit this grading to simplify the computation of the first few terms of the linear strand of the minimal free resolution of $M$.

Let $u = (u_1, u_2) \in \Ker \dd_1 \subseteq R_G(-\vec{e}_2) \oplus R_G(-\vec{e}_4)$ be a nonzero homogeneous element of degree $\vec{a} \in \ZZ^5$.  Since $(R_G)_\vec{b} = 0$ if any component of $\vec{b}$ is negative, we must have either $\vec{e}_2 \leq \vec{a}$ or $\vec{e}_4 \leq \vec{a}$, where $\vec{b} = (b_1, \dots, b_5) \leq (a_1, \dots , a_5) = \vec{a}$ means that $b_i \leq a_i$ for all $i$.  If $u$ is a linear syzygy, then we  have $\abs{\vec{a}} = \sum_i a_i = 2$, so  $\vec{a} = \vec{e}_2 + \vec{e}_i, \vec{e}_4 + \vec{e}_j, \vec{e}_2 + \vec{e}_4$ for some $i \neq 4$ and $j \neq 2$.  In the first case, we have $u = (\alpha x_i + \beta y_i, 0)$ for some $\alpha, \beta \in \kk$, so $\dd_1u = 0$ implies $\alpha x_2x_i + \beta x_2y_i = 0$.  As $x_2x_i$ is part of the monomial basis for $R_G$, it follows that $\alpha = 0$, and so  $\beta x_2y_i = 0$. If $i \neq 5$, then $0 = \beta x_2y_i = \beta x_iy_2$.  In any case, either $x_2y_i$ or $x_iy_2$ is part of the monomial basis for $R_G$, so  $\beta = 0$, which contradicts that $u$ is nontrivial.  A similar argument shows that there are no nontrivial elements of $\Ker \dd_1$ in degree $\vec{a} = \vec{e}_4 + \vec{e}_j$ for $j \neq 2$.  However, when $\vec{a} = \vec{e}_2 + \vec{e}_4$, we see that $u = (\alpha_1x_4 + \beta_1y_4, \alpha_2x_2 + \beta_2y_2)$, so $\dd_1u = 0$ implies 
\[
0 = \alpha_1x_2x_4 + \beta_1x_2y_4 + \alpha_2x_2x_4 + \beta_2x_4y_2 = (\alpha_1 + \alpha_2)x_2x_4 + (\beta_2 + \beta_1)x_4y_2.
\]
As $x_2x_4$ and $x_4y_2$ are part of the monomial basis for $R_G$, it follows that $\alpha_2 = -\alpha_1$ and $\beta_2 = -\beta_1,$ so 
\[
u = \alpha_1\begin{pmatrix}
    x_4 \\ -x_2
\end{pmatrix} + \beta_1\begin{pmatrix}
    y_4 \\ -y_2
\end{pmatrix},
\]
and the columns of the matrix
\[
\dd_2 = \begin{pmatrix}
    x_4 & y_4 \\ -x_2 & -y_2
\end{pmatrix}
\]
are easily seen to generate the linear syzygies of $\dd_1$.

If $u = (u_1, u_2) \in \Ker \dd_2 \subseteq R(-\vec{e}_2-\vec{e}_4)^2$ is a nonzero homogeneous element of degree $\vec{a}$, then $\vec{a} \geq \vec{e}_2 + \vec{e}_4$.  If in addition $u$ is a linear syzygy, then $\vec{a} = \vec{e}_2 + \vec{e}_4 + \vec{e}_i$ for some $i$, so  $u = (\alpha_1x_i + \beta_1y_i, \alpha_2x_i + \beta_2y_i)$.  Then $\dd_2u = 0$ implies: 
\begin{align*}
0 &= \alpha_1x_ix_4 + \beta_1x_4y_i + \alpha_2x_iy_4 + \beta_2y_iy_4
\\
0 &= \alpha_1x_ix_2 + \beta_1x_2y_i + \alpha_2x_iy_2 + \beta_2y_iy_2    
\end{align*}
As $x_ix_4$ and $y_iy_4$ are part of the monomial basis for $R_G$, it follows that $\alpha_1 = \beta_2 = 0$, so $\beta_1x_4y_i + \alpha_2x_iy_4 = 0$ and $\beta_1x_2y_i + \alpha_2x_iy_2 = 0$.  If $i = 1, 5$, then $\alpha_2 = \beta_1 = 0$ as well since $x_4y_1, x_1y_4, x_2y_5$, and $x_5y_2$ are also part of the monomial basis for $R_G$, which contradicts that $u$ is nontrivial.  Hence, we must have $i = 2, 3, 4$, in which case we have $0 = \beta_1x_4y_i + \alpha_2x_iy_4 = (\beta_1 + \alpha_2)x_iy_4 = (\beta_1 + \alpha_2)x_4y_i$ where either $x_iy_4$ or $x_4y_i$ is part of the monomial basis for $R_G$, so  $\alpha_2 = -\beta_1$.  From this, it is easily see that the columns of the matrix
\[
\dd_3 = \begin{pmatrix}
    y_2 & y_4 & y_3 \\ -x_2 & -x_4 & -x_3
\end{pmatrix}
\]
generate the linear syzygies of $\dd_2$.

Let $u = (u_1, u_2, u_3) \in \Ker \dd_3 \subseteq R(-2\vec{e}_2-\vec{e}_4) \oplus R(-\vec{e}_2-2\vec{e}_4) \oplus R(-\vec{e}_2-\vec{e}_3-\vec{e}_4)$ be a nonzero linear syzygy of degree $\vec{a}$.  If $\vec{a} \ngeq \vec{e}_2 + \vec{e}_3 + \vec{e}_4$, then $u_3 = 0$, so  $(u_2, u_1) \in \Ker \dd_1$ and
\[ 
u = \alpha\begin{pmatrix}
    -x_2 \\ x_4 \\ 0
\end{pmatrix} +
\beta\begin{pmatrix}
    -y_2 \\ y_4 \\ 0
\end{pmatrix}
\]
for some $\alpha, \beta \in \kk$.  So, we may assume that $\vec{a} \geq \vec{e}_2 + \vec{e}_3 + \vec{e}_4$, in which case we have either:
\begin{enumerate}
    \item $\vec{a} = \vec{e}_2 +\vec{e}_3 + \vec{e}_4 + \vec{e}_i$ for some $i \neq 2, 4$,
    \item $\vec{a} = \vec{e}_2 +\vec{e}_3 + 2\vec{e}_4$, or 
    \item $\vec{a} = 2\vec{e}_2 +\vec{e}_3 + \vec{e}_4$. 
\end{enumerate}
In case (1), we have $u = (0, 0, \alpha x_i + \beta y_i)$ for some $\alpha, \beta \in \kk$, so  $\dd_3u = 0$ implies $0 = \alpha x_iy_3 + \beta y_iy_3$ and $0 = \alpha x_ix_3 + \beta y_ix_3$.  As $x_ix_3$ and $y_iy_3$ are part of the monomial basis for $R_G$, it follows that $\alpha = \beta = 0$, which contradicts that $u$ is nontrivial.  In case (2), we have $u = (0, \alpha_2x_3 + \beta_2 x_3, \alpha_3x_4 + \beta_3 x_4)$ for some $\alpha_i, \beta_i \in \kk$. Then $\dd_3u = 0$ implies 
\[
0 = \alpha_2x_3y_4 + \beta_2y_3y_4 + \alpha_3x_4y_3 + \beta_3y_3y_4 =  (\alpha_2 + \alpha_3)x_4y_3 + (\beta_2 + \beta_3)y_3y_4.
\]
As $x_4y_3$ and $y_3y_4$ are part of the monomial basis for $R_G$, we see that $\alpha_3 = -\alpha_2$ and $\beta_3 = -\beta_2$, so 
\[ 
u = \alpha_2\begin{pmatrix}
    0 \\ x_3 \\ -x_4
\end{pmatrix} +
\beta_2\begin{pmatrix}
    0 \\ y_3 \\ -y_4
\end{pmatrix}.
\]
By a similar argument applied to case (3), we see that the columns of the matrix
\[
\dd_4 = \begin{pmatrix}
    -x_2 & -y_2 & x_3 & y_3 & 0 & 0 \\
    x_4 & y_4 & 0 & 0 & x_3 & y_3 \\
    0 & 0 & -x_2 & -y_2 & -x_4 & -y_4
\end{pmatrix}
\]
generate the linear syzygies of $\dd_3$.

Finally, we observe that $x_3f_{1,5} = x_1f_{3,5} + x_5f_{1,3} \in J_G$ and similarly that $y_3f_{1,5} \in J_G$, so  $q = (0, 0, f_{1,5})$ is a quadratic syzygy of $\dd_3$.  If $q$ is a non-minimal syzygy, the above description of the linear syzygies of $\dd_3$ implies that $f_{1,5} \in J_G + (x_2, y_2, x_4, y_4)$, which is easily seen to imply that $f_{1,5} \in J_P$ where $P$ is the induced path of $G$ on the vertices $\{1,3, 5\}$.  Since the vertex labeling of $P$ is closed, we see that $\init_>(J_P) = (x_1y_3, x_3y_5)$, but then $\init_>(f_{1,5}) = x_1y_5 \notin \init_>(J_P)$, which is a contradiction.  Therefore, $q$ is a minimal quadratic syzygy of $\dd_3$, hence also of $M$. Hence, $M$ does not have a linear resolution over $R_G$, and $R_T$ is not Koszul.
\end{proof}

\section{Strongly Chordal Claw-Free Graphs with Small Clique Number}\label{thicknet}

\noindent In this section, we characterize strongly chordal, claw-free graphs with small clique number.  The \term{clique number} $\omega(G)$ of a graph $G$ is the maximum size of a clique of $G$.  To state our characterization, we need a little more terminology.  A \term{cut vertex} of a connected graph $G$ is a vertex $v$ such that $G \setminus v$ is not connected.  A graph $G$ is \term{2-connected} if it has no cut vertices.  A \term{block} of a graph $G$ is a maximal 2-connected subgraph.

As an immediate combinatorial consequence of our main theorem~\ref{mainthm} and \cite[2.6]{EHH}, we have the following.

\begin{cor}
Let $G$ be a connected graph with $\omega(G) \leq 3$.  The following are equivalent:
\begin{enumerate}[label = \textnormal{(\alph*)}]
    \item $G$ is strongly chordal and claw-free. 
    \item There is a tree $T$ in which every vertex has degree at most 3 and a family of induced closed subgraphs $\{G_v\}_{v \in T}$ of $G$ such that:
\begin{enumerate}[label = \textnormal{(\roman*)}]
    \item $V(G_v) \cap V(G_w) = \{i\}$ is a cutset of $G$ whenever $\{v,w\}$ is an edge of $T$.
    \item If $u$ and $w$ are distinct neighbors of $v$ in $T$, then $V(G_u) \cap V(G_v) \cap V(G_w) = \emptyset$.
    \item If $v$ is a degree 3 vertex, then $G_v$ is a clique of size 3.
\end{enumerate}
\item $G$ is claw-free, and every block of $G$ is a closed graph.
\end{enumerate}
\end{cor}

\begin{proof}
    (a) $\Iff$ (b): This follows immediately from our main theorem~\ref{mainthm}  and \cite[2.6]{EHH}.

    (b) $\Implies$ (c): If $G$ can be decomposed as described in part (b), then $G$ is claw-free by the equivalence of (b) and (a), and the blocks of $G$ are just the blocks of $G_v$ for all $v$.  Since induced subgraphs of closed graphs are closed by  \cite[2.2]{closed:graphs:are:proper:interval:graphs}, it follows that every block of $G$ is closed.

    (c) $\Implies$ (a): To see that $G$ is strongly chordal when it is claw-free and all of its blocks are closed, it suffices by our main theorem~\ref{mainthm} to note that $G$ is chordal and tent-free.  Since cycles and tents are 2-connected, any induced cycle or tent of $G$ would be contained in one of its blocks.  Thus, $G$ is chordal and tent-free since the blocks of $G$ are closed. 
\end{proof}

We note that the implication (c) $\Implies$ (a) of the preceding corollary is always valid without any assumption on $\omega(G)$. Unfortunately, the following example shows that the decomposition of the above corollary and the implication (a) $\Implies$ (c) do not carry over as stated to higher clique numbers. 

\begin{example} 
Consider the \term{thick net} graph $TN$ shown in Figure~\ref{thick:net}.
\begin{figure}[htb]
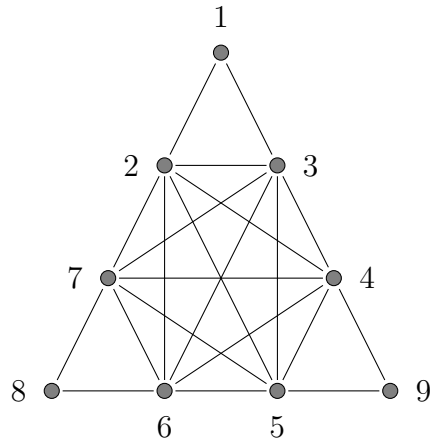

\begin{center}
\graph{
    -1/0/6/-90,
    -1/4/2/180,
    1/4/3/0,
    1/0/5/-90,
    -3/0/8/180,
    0/6/1/90,
    3/0/9/0,
    -2/2/7/180,
    2/2/4/0
    }{
    6/2,
    2/3,
    3/5,
    6/5,
    2/5,
    6/3,
    8/7,
    8/6,
    1/2,
    1/3,
    9/5,
    9/4,
    7/6,
    2/7,
    3/7,
    4/7,
    6/4,
    2/4,
    3/4,
    4/5,
    5/7}
\end{center}
\caption{The thick net}
\label{thick:net}
\end{figure}
It is easily checked that $TN$ is 2-connected, but $TN$ is not a closed graph since the induced subgraph on the vertices $\{1, 2, 4, 7, 8, 9\}$ is a net.  We show that $TN$ is nonetheless strongly chordal and claw-free.  If we set $e_1 = \{5,9\}$ and $e_2 = \{4, 9\}$, it is easily seen that $e_1$ is a claw-avoiding simplicial edge for $TN$ and $e_2$ is a claw-avoiding simplicial edge for $TN \setminus e_1$.  Since $TN \setminus \{e_1, e_2\}$ is a closed graph for the given labeling, it follows from Corollary \ref{strongly:chordal:simplicial:edge:deletion} that $TN$ is strongly chordal and claw-free.
\end{example}

It is an open question to
describe the structure of claw-free strongly chordal graphs with $\omega(G) \geq 4$.   In particular, if $\omega(G) \leq k$ for some $k \geq 2$,  we don't know if being strongly chordal and claw-free is equivalent to being chordal and every maximal $(k-1)$-connected subgraph of $G$ is closed;
note that the thick net does not rule out this, since its maximal 3-connected subgraphs are  its maximal cliques.

\paragraph{Acknowledgments}
The authors thank Arvind Kumar for some helpful conversations related to this work.  Computations with Macaulay2 \cite{M2} were very useful while working on this project.
LaClair was partially supported by National Science Foundation grant DMS--2100288 and by Simons Foundation Collaboration Grant for Mathematicians \#580839.  Mastroeni was partially supported by an AMS-Simons Travel Grant.  McCullough was partially supported by National Science Foundation grants DMS--1900792 and  DMS--2401256.  Peeva was partially supported by National Science Foundation grant DMS-2401238.

\newcommand{\etalchar}[1]{$^{#1}$}

\end{document}